\documentclass [10pt] {article}
\author{Matthieu Arfeux}

\usepackage[latin1]{inputenc}%
\usepackage{amssymb,amsfonts,epsfig,amsmath,graphics,graphicx}%
\usepackage{amsthm}%
\usepackage[all]{xy}%
\usepackage{verbatim}%
\usepackage{enumitem}
\usepackage{bm}
\usepackage[nottoc, notlof, notlot]{tocbibind}

\title {Reading escaping trees from Hubbard trees in $\St_n$}
\date{}

\newtheorem{theorem} {Theorem}[section]
\newtheorem{proposition}[theorem]{Proposition}
\newtheorem{lemma}[theorem]{Lemma}
\newtheorem*{lemma*}{Lemma}
\newtheorem{corollary}[theorem]{Corollary}
\newtheorem{definition}[theorem]{Definition}

\newtheorem*{conjecture*}{Conjecture}

\newtheorem{example}[theorem]{Example}

\theoremstyle{plain}
\newtheorem{theoremint} {Theorem}

\theoremstyle{remark}
\newtheorem{remark}[theorem]{Remark}

\renewenvironment{proof}{\noindent{\bf Proof. }}{\hfill{$\square$} \vskip.3cm}

\def\cal{\mathcal}

\def\C{{\mathbb C}}
\def\D{{\mathbb D}}

\def\N{{\mathbb N}}

\def\Q{{\mathbb Q}}
\def\R{{{\mathbb R}}}

\def\St{{\mathcal S}}

\def\Z{{\mathbb Z}}

\def\Per{{\rm Per}}

 \def\epsilon{{\varepsilon}}

\def\intertw{\angle \hskip -7pt \text{\reflectbox{$\angle$}}}

\usepackage{hyperref}

\hypersetup{
backref=true, 
pagebackref=true,
hyperindex=true, 
colorlinks=true, 
breaklinks=true, 
urlcolor= blue, 
linkcolor= blue, 
}

\begin{document}

\maketitle

\begin{abstract}

We prove that the parameter space of monic centered cubic polynomials with a critical point of exact period $n=4$ is connected. The techniques developed for this proof work for every $n$ and provide an interesting relation between escaping trees of DeMarco-McMullen and Hubbard trees. 

\end{abstract}

\medskip
\noindent{\textbf{About this preprint.}} This preprint is based on an expected result of current work by Cui Guizhen. Hence, it is not in its final version.



\section{Introduction}

Consider the space of monic centered cubic polynomials of the form
$$P_{a,b}(z):=z^3-3a^2z+b,$$
where $(a,b)\in\C^2$. All of these polynomials have critical points at $a$ and $-a$. Following \cite{CubPol1} and \cite{CubPol2},we are interested in the subspace $\Per^{cm}_n$ consisting of the parameters for which  the polynomial has a critical point of exact period $n$. This space naturally splits as the union $$\Per^{cm}_n=\St_n\cup\St_n^-,$$
where $a$ is periodic for points in $\St_n$ and $-a$ for those in $\St_n^-$. 

We are interested in the following conjecture.

\begin{conjecture*}[Milnor] The set $\St_n$ is irreducible for every $n\geq1$.
\end{conjecture*}

In \cite{CubPol1}, it is proven that these curves are smooth and it remains to prove that they are connected. This result was proven in \cite{CubPol1} for $n=1,n=2$ and $n=3$. It was observed in \cite{CubPol2} for $n=4$ by computing local pictures of this curve and observing after gluing them together that the space is connected. This paper describes a non-computer-assisted method for proving that $\St_4$ is connected so we can deduce the following:

\begin{theoremint}\label{main}
 The space $\St_4$ is irreducible.
\end{theoremint}

Recall that the dynamical plane of a polynomial $P$ splits into two pieces:
\begin{itemize}
\item the filled Julia set $K_P$ consisting of the set of points that are not attracted to infinity, and
\item the basin of attraction of infinity $\Omega_\infty$ which is exactly the complement of $K_P$.
\end{itemize}
A cubic polynomial $P\in \St_n$ has at most two (finite) critical points, so from the point of view of the parameter space we have the dichotomy: 
\begin{itemize}
\item either the two critical points of $P$ have bounded orbits under iteration and the filled Julia set of $P$ is connected,
\item or the fllled Julia set of $P$ is disconnected.
\end{itemize}
According to this remark, we define ${\cal C}$ to be the connectedness locus of $\St_n$, the space of polynomial whose filled Julia set is connected. We call escape component each connected component of its complement ${\cal E}:=\C\setminus {\cal C}$.
The connectedness locus is known to be compact. Since irreducible components of algebraic curves are unbounded, we just have to prove that a connected component of $\St_n$ contains all the escape components. 

We are going to prove the existence of points at the boundary of different escape components.
More precisely, we consider some polynomials in ${\cal C}$ obtained by intertwining a parabolic polynomial and a super-attractive polynomial. The \emph{intertwining} is a surgery introduced by A.Epstein and M.Yampolsky in \cite{EY}, generalized to this case by P.Haïssinsky in \cite{Ha} that consist in creating a cubic polynomial from two quadratic ones.
Using a current work of Cui G. \cite{C2} we perturb them in the direction of different  escape components and we prove that these perturbations are sufficiently controlled so that we can recognize the escape components we are moving to. 

For this we need to label the escape components. We associate to each one a pair constituted by a Hubbard tree and an escaping tree. On the one hand, \emph{Hubbard trees} were introduced in \cite{Orsay} (see also \cite{Poirier}) to study connected fielled Julia sets and on the other hand, \emph{escaping trees} have been developed to study the basin of attraction of infinity in \cite{Treesph} and improved in a long series of paper by L.DeMarco and K.Pilgrim (\cite{DP3}, \cite{DP1},\cite{DP2},\cite{DP4}, see also L.DeMarco and A.Schiff \cite{DS1},\cite{DS2}). 

We then conclude the proof by finding sufficiently enough polynomials in the case of $n=4$ to connect all the escape components.

The methods developed in this paper are valid for all periods but they prove the conjecture only in the case of period 4, when the labeling is one to one. One may expect that these methods with a little more work give the result for higher periods but the combinatorics that we would have to check would quickly leave the human level.
However, the novelty is here to \emph{describe a bridge between Hubbard trees and escaping trees}. Those are completely different kinds of trees not only from their construction but also from the kind of question they were introduced for. We will prove that given a Hubbard tree of a quadratic polynomial, we can construct escaping trees of a cubic polynomial.

\medskip
\noindent{\textbf{Outline.}}
In Section \ref{Chap2}, we recall the definitions of escaping trees, Hubbard trees, kneading sequences and some of their basic properties. We introduce the notion of an escaping lamination and prove that it contains at least the informations of an escaping tree.

In section \ref{Chap3}, after recalling some basic tools about rays in the parameter space and in the dynamical plane, we look at polynomials in $\St_n$ with a Fatou component of parabolic type. We construct from each of these polynomial a lamination that looks like an escaping lamination and we prove that it is the escaping lamination of the perturbation of the polynomial. We also prove that we can find the Hubbard tree of the perturbed polynomials. 

In section \ref{Chap4}, we recall briefly the construction of intertwinings and we deduce the existence of some common boundary points between escape components. We give a list of such points in the cases $n=4$ and $n=5$. These points are sufficient to prove Theorem \ref{main}. We then explain how we construct an escaping tree of a cubic polynomial from the Hubbard tree of a quadratic polynomial. We finally discuss about problems for generalizing these techniques for proving Milnor's conjecture for $n\geq5$.

\medskip
\noindent{\textbf{Acknowledgments.}} This paper is the result of a lot of different discussions. I would want to thank more especially Cui Guizhen, Laura DeMarco and Jan Kiwi. I would also want to thank Araceli Bonifant and John Milnor for their attentive listening and precious advice.


\section{Escaping trees and Hubbard trees}\label{Chap2}


\subsection{ Green function and escaping tree.}

Given a cubic polynomial $P$, there is a natural map $g_P:\C\to[0,\infty)$ defined by $$g_P(z):=\lim_{n\to\infty}\frac{1}{3^n}\log^+|P^n(z)|,$$ called the Green function which is identically equal to 0 on the filled Julia set $K_P$ and satisfies $g_p(P(z))=3\cdot g_p(z)$ (see \cite{DynInOne} for example). Topologically, the non zero level curves of the Green function are unions of circles which can intersect only at backward preimages of the escaping critical points of $P$, which are exactly the critical points of $g_P$. These curves are called critical level curves.

In \cite{Treesph}, L. De Marco and C. McMullen introduced a natural real tree associated to a cubic polynomial, obtained after collapsing each connected component of all the Green levels curves to points. Edges correspond to annuli made of level curves and vertices correspond to grand orbits of critical level curves.
The points corresponding to connected components of the filled Julia set are called end points of the tree. 
 In this paper we will just consider this tree as a combinatorial tree. 
 
 This tree comes with a natural dynamic on it; indeed every polynomial maps a connected component of Green's level curve to an other one. This dynamical tree will be called an escaping tree and it is invariant in an escape component.

We are interested in such trees when the polynomials are in $\St_n$. In \cite{DP4}, L. De Marco and K. Pilgrim proved the following:

\begin{proposition}\label{restricyiuov}
 The subtree consisting of the convex hull of the ends containing the critical point orbit together with the corresponding dynamics on the ends is sufficient to recover the all escaping tree.
 \end{proposition}
 We will call simplified escaping tree such a tree together with its restricted dynamics.


\subsection{Böttcher coordinates and escaping lamination.}\label{labelle}

Recall that every $P\in\St_n$ has a critical point $a$ of exact period $n$ and a second critical point at $-a$. Define $r:=e^{g_P(-a)}$.

There exist a biholomorphism 
$$\phi_P:\{g_P>g_P(c_P)\}\to\{ z\in\C:|z|> r\}$$
 that satisfies $\phi_P\circ P=(\phi_P)^3$ and $g_P=log|\phi_P|$  (see \cite{DynInOne} for example). It is unique up to multiplication by a third root of unity. We chose for $\phi_P$ the function that is tangent to identity close to infinity. This map called the Böttcher coordinate.

For $\theta \in\R/2\pi\Z$, we call (dynamical) ray of angle $\theta$ the set 

$${\cal R}_\theta:=\phi_P^{-1}(\{\rho e^{i\theta}:\rho>r\}).$$ This set correspond to the gradient flow line of $g_P$ on $\{g_P>r\}$ and can be uniquely extended to $\Omega_\infty$ as soon as the trajectory does not meet the backward orbit of $-a$ which correspond exactly to the critical points of $g_P$. 

For $k\in\N$ we define $\Omega_k:=\{3^k.g_P>r\}$ and denote by ${\cal R}^k_\theta$ and call {$k$-generalized} ray of angle $\theta$ the union of all possible such extensions on $\overline{\Omega_k}$. We define an equivalence relation on $\R/2\pi\Z$ defined as follows: 
$$\theta_1\sim_k\theta_2\quad\iff\quad {\cal R}^k_{\theta_1}\cap {\cal R}^k_{\theta_2}\neq \emptyset.$$

\begin{example}\label{lemni}For $k=0$ and $-a\in \Omega_\infty$, the unique class which is non trivial contains exactly the two angles corresponding to the generalized rays containing $c_P$. 
\end{example}

\begin{remark}\label{dynrel}
A $k$ class is always contained in a $k+1$ class, and the elements of a $k+1$ class are mapped by multiplication by 3 to the elements of a $k$ class.
\end{remark}

We define the $k$-escaping critical portrait to be ${\bf \Theta}^k_P$ the set of non trivial classes of the relation $\sim_k$. We denote by $\D$ the unit disc of $\C$ equipped with the hyperbolic metric.

\begin{definition}[Escaping lamination]
  Given any $k\in\N$ and a $k$-escaping critical portrait ${\bf \Theta}^k_P=\{\Theta_1,\ldots,\Theta_p\}$, we define the $k$-lamination $\mathcal{L}_k$
to be the subset of ${\C}$ obtained as the union of the convex hull of ${\Theta_i \subset \R/2\pi\Z \equiv \partial \D}$ inside $\D$.

We call $k$-gap every connected component of the complement of $\mathcal{L}_k$ in $\D$ and denote by $\mathcal{G}_k$ the set of $k$-gaps.
\end{definition}

All of these definitions extends naturally to the case $k\in\Z$. For $k<0$ we have $\mathcal{L}_k=\emptyset$ and $\mathcal{G}_k=\{\D\}$. An example of lamination is sketched on figure \ref{lamin}.

It follows from these definitions that 
$$\mathcal{L}_{k}\subset \mathcal{L}_{k+1}\quad\text{and}\quad \mathcal{G}_k\subset\mathcal{G}_{k+1}.$$


We remark that $\Omega_\infty=\bigcup\Omega_k$. Let $B_k$ be the set of connected components of $\Omega_{k}\setminus\overline{\Omega_{k-1}}$.
There is a natural bijection $$\phi_k:\mathcal{G}_k\to B_k,$$ 
that to $\gamma_k\in \mathcal{G}_k$, associate the unique $b_k\in B_k$ such that $${\cal R}_\theta^k\cap b_k\neq\emptyset\iff\theta\in\partial \gamma_k\cap\partial\D.$$
Note that if $\gamma_k$ is bounded by $\Theta_1,\ldots,\Theta_p\in{\bf \Theta}^k_P$  then $\phi_k(\gamma_k)$ is included in a component of $\Omega_{k}$ bounded by the corresponding generalized rays.

 As $P$ defines a map from $B_{k+1}$ to $B_k$, we define by conjugacy a map ${F_{k+1}:\mathcal{G}_{k+1}\to \mathcal{G}_k}$ i.e. such that the following diagram commutes:
 
 \centerline{$\xymatrix{
    B_{k+1} \ar[r]^P   & B_k  \\
   \mathcal{G}_{k+1} \ar[u]^{\phi_{k+1}}\ar[r]^{F_{k+1}} & \mathcal{G}_{k}\ar[u]_{\phi_k}.
  }$}

Similarly, following Remark \ref{dynrel}, we define a map $F_{k+1}:\mathcal{L}_{k+1}\to \mathcal{L}_k$. 
We have defined for every $k\geq0$ a map $$F_{k+1}:\mathcal{G}_{k+1}\sqcup\mathcal{L}_{k+1}\to \mathcal{G}_k\sqcup \mathcal{L}_k.$$

\begin{definition}
The dynamical escaping lamination is the family $(\mathcal{L}_k,F_k)_{k\in\Z}$.
\end{definition}



\subsection{Escaping tree and escaping lamination.}

In this subsection we prove the following theorem:

\begin{proposition}\label{treelam}
The escaping lamination associated to a polynomial determines its escaping tree.
\end{proposition}

 \begin{figure}
  \centerline{\includegraphics[width=10cm]{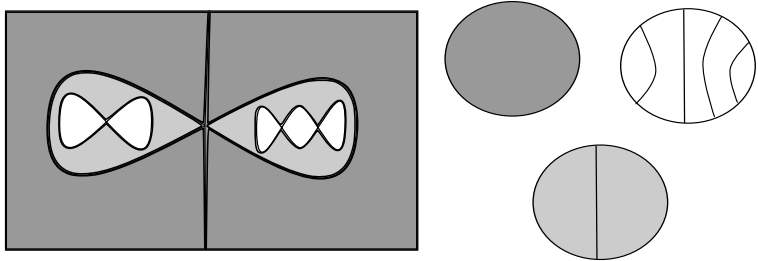}} 
   \caption{On the left, the upper vertical line is the ray $R^0_\theta$ where $\theta$ lies in the unique non trivial class of $\sim_0$. The dark grey region is $\Omega_0$. The two light grey regions corresponds to the two elements of $B_0$. On the right the corresponding laminations $\mathcal{L}_k$ for $k=-1,0,1.$}
\label{lamin} 
\end{figure}

We remark that ${\bf\mathcal{G}}$ is a filtered set with respect to the inclusion order, i.e. 
\begin{itemize}
\item it is a partially ordered set, and
\item if $\gamma_1,\gamma_2\in {\bf\mathcal{G}}$ then there exists $\gamma_3\in {\bf\mathcal{G}}$ such that $\gamma_1\subset\gamma_3$ and $\gamma_2\subset\gamma_3$.
\end{itemize}
So this set has a natural associated tree and this one has a dual tree. More precisely, the edges of this tree are the elements of ${\bf\mathcal{G}}$ and there is a vertex between two edges if and only if there exists $k$ such that the gaps are a $(k+1)$-gap and a $k$-gap and the first one is included in the other. 

Such a tree is the combinatorial escaping tree associated to this polynomial. Indeed, in one hand, the elements of the $B_k$ are annuli constituted by a disjoint union of circles which are connected components of levels curves of the Green function. On the other hand, by construction, for $\gamma_k\in \mathcal{G}_k$ and $\gamma_{k+1}\in \mathcal{G}_{k+1}$, we have ${\partial \phi_k(\gamma_k)\cap\partial \phi_{k+1}(\gamma_{k+1})\neq\emptyset}$ if and only if $\gamma_{k+1}\subset\gamma_{k}$.

We define the dynamics on the tree according to the maps $F_k$. With the identification of the escaping tree, it is clear that the dynamics on the two trees correspond.



 \begin{figure}
  \centerline{\includegraphics[width=7cm]{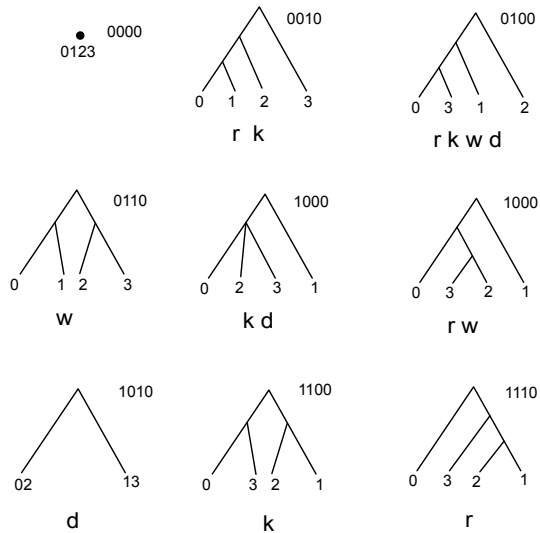}} 
   \caption{All possible simplified escaping trees for polynomials in $\St_4$ with their kneading sequence. We nicknamed four trivial escaping components as in \cite{CubPol2} (k for kokopelli, w for worm, d for  double-basilica and r for 1/4-rabbit). Below each non trivial tree we wrote the nickname of each of these trivial escape components at its boundary.}
\label{class4} 
\end{figure}

\begin{definition}
The simplified dynamical escaping lamination is a escaping lamination for which we just consider the forward dynamic of the gaps containing the critical points.
\end{definition}

Then from Proposition \ref{restricyiuov} we deduce:

\begin{corollary}\label{reconst}
Given a simplified escaping lamination associated to a polynomial, we can reconstruct its dynamical escaping lamination.
\end{corollary}



\subsection{Kneading sequence.}

We recall the natural kneading sequence associated to a polynomial $P\in{\cal E}$, already introduced in \cite{CubPol1} and \cite{CubPol2}. The Green's function level curve containing the escaping critical point is topologically a Bernoulli's lemniscate. The Julia set is contained in the bounded connected components of its complementary. Label by $V_0$ the one that contains the other critical point and by $V_1$ the other. There is a natural labeling of the connected components of the filled Julia set according the itinerary of its elements with respect to this partition.

Let $K$ be the connected component of $K_P$ containing the critical $a$. We define the \emph{kneading sequence} $(k_j)_{j\in\N}$ of $P$  by $k_j=i$ if and only if $P^{j}(K)\subset V_i$. As $P\in\St_n$ this sequence is periodic with exact period dividing $n$. We simplify the notation writing only the $n$ first elements of this sequence.
 
The kneading sequence can be read on the escaping tree. Indeed, it follows from the previous remarks that the escaping tree minus the vertex corresponding to the Green's function level curve containing the escaping critical point is a disjoint union of three branches, two of them correspond exactly to the level curves included in $V_0$ and $V_1$. As $V_0$ and $V_1$ are in natural correspondence with the 0-gaps, we can do a similar remark for the escaping lamination and polynomials lying in the same escape component have the same kneading sequence.

\begin{example}When $n=4$, we can have three different periods $p$ for the kneading sequence which implies three different behavior for the first return map in the connected component of the filled Julia set of the periodic critical point:
\begin{itemize}
\item $p=4$ then the return map fixes the critical point,
\item $p=2$ then the return map has a critical point of period 2, or
\item $p=1$ then the return map has a critical point of period 4.
\end{itemize}
\end{example}

In \cite{CubPol2}, the authors remarked that the kneading sequence is not enough to characterize the basin of infinity of polynomials in $\St_4$. In Figure \ref{class4}, we see that there are indeed two trees with the same kneading sequence.


\subsection{Hubbard Tree and Labeling escaping components.}

Let $P$ be a polynomial in $\St_n$ with disconnected Julia set. Suppose that the connected component of $K_P$ containing $a$ has period $k\leq n$. It follows form \cite{BH1} and \cite{BH2} that for a well chosen neighborhood of $a$, the restriction $P^k$ is a polynomial-like map of degree 2.  There is a unique representative of its hybrid class in the family $p_c:z\mapsto z^2+c$ with $c\in\C$ and this one has a critical point of exact period $n/k$. We call this representative the \emph{connected restriction} of $P$, it is an invariant for the topological conjugacy class of $P$, ie every polynomial in the same escape component of $\St_n$ have the same connected restriction. Hence, we will talk about the connected restriction of an escape component.

Let $p_c$ be such a connected restriction. It has a periodic Fatou component $U$ containing its unique critical point and every other bounded Fatou component eventually maps on this cycle. It follows that in $U$ there is a well defined Bötcher coordinate and dynamical rays that can be pulled back to its backward orbit. 

The convex-hull of the critical orbit inside the filled Julia set crossing the Fatou components through dynamical rays form a topological tree that characterize $p_c$ called its Hubbard tree (cf \cite{Orsay} and \cite{Poirier}). We will call it the Hubbard tree of the escape component.

Then, as explained in $\cite{DS2}$ \textsection5.3 for example, we can state the following: 

\begin{proposition}\label{BHclass}
The topological conjugacy class of an escape component of $\St_n$ is determined by 
\begin{itemize} 
\item the topological conjugacy class of the basin of infinity and
\item its Hubbard tree. 
\end{itemize}
\end{proposition}

It follows from the work of L.De Marco and K. Pilgrim \cite{DP4}, that for $n=4$, the escaping tree is enough to characterize the topological conjugacy class of the basin of infinity so we have the following result :

\begin{proposition}\label{DPclass}
For $n=4$, every escape component is characterized by its Hubbard tree together with its escaping tree.
\end{proposition}

Thus it follows from Proposition \ref{treelam} and Corollary \ref{reconst} that:
\begin{corollary}
The escape component of any escaping polynomial in $\St_4$ is characterized by its Hubbard tree together with its simplified escaping lamination.
\end{corollary}

In the general period case, the situation is more complicated and in the same paper they introduced various combinatorial tools to try to encode the topological conjugacy class of the basin of infinity. This will be discussed later at the end of section \ref{reading}.
Figure \ref{class4} represents all the possible simplified trees for polynomials in $\St_4$. On that figure we recalled the corresponding kneading sequences.




\section{Parabolic polynomials and controlled perturbations}\label{Chap3}

\subsection{Rays and parameter space.}

In \cite{CubPol1}5.9, J.Milnor proves that for every escape component ${\cal E}$, the map $\Phi:{\cal E}\to\C\setminus \D$ defined by $\Phi(a,b)=\phi_{P_{a,b}}(2a)$ is a covering. Hence escape components are conformaly a punctured disc and there is a well defined notion of equipotentials and external angles.

The consequence on the dynamics of the polynomials of moving along these curves is well understood, it corresponds for external angles to "stretch" the dynamic on the basin of infinity and for equipotential to twist the dynamics in the annulus between the Green level curves of the escaping critical point and of the critical value. Hence in the first case, the argument of $\phi_{P_{a,b}}(2a)$ stay unchanged whereas its modulus is fixed in the second case. We deduce the following lemma:

\begin{lemma}\label{sameangle}
Any path $(a_t,b_t)$ inside an escape component on which the argument of $\phi_{P_{a_t,b_t}}(2a_t)$ is constant is included in an external angle.
\end{lemma}

Usually the argument of the cocritical point of a polynomial is not constant after perturbation. However there are properties about angles that stay unchanged in a neighborhood of a polynomial. Indeed, recall that we say that an (extended) dynamical ray $(\phi_P^{-1}(re^{i\theta}))_{r>0}$ lands to a point $z_0$ of the Julia set if 
we have $$z_0=\lim_{r\to0}\phi_P^{-1}(re^{i\theta});$$ we have the following result:

\begin{lemma} \cite{GM}\label{Stability}B1.
For any polynomial $P$, suppose that $z_0$ is a repelling fixed point of $P$, and suppose that some rational external ray $(\phi_P^{-1}(re^{i\theta}))_{r>0}$ lands at $z_0$. Then for any $Q$ sufficiently close to $P$, the corresponding ray $(\phi_{Q}^{-1}(re^{i\theta}))_{r>0}$ lands at the corresponding fixed point of $Q$.
\end{lemma}

Remark that this lemma does not require any condition on the polynomial, i.e. $P$ can lie in the connectedness locus. Here is the main tool of our paper that allows us to go from one escape component to an other one but keeping informations when we cross ${\cal C}$.

\subsection{Lamination of a perturbation.}\label{sectionparab}

We are interested in polynomials in $\St_n$ which have a parabolic periodic point. We will say that they are \emph{parabolic polynomials}.
In this section $P$ is such a polynomial. We are going to define some basic vocabulary and some lamination associated to $P$. Later, this lamination will be recognized as the escaping laminations of some polynomials close to $P$.

The polynomial $P$ has exactly one super-attractive cycle of Fatou components and one parabolic one. Each of these cycles contains exactly one critical point  (cf \cite{DynInOne} for example). Denote by $V_{-a}$ the Fatou component containing $-a$. 
A dynamical ray of $P$ is said \emph{parabolic} if it lands at a parabolic periodic point in the boundary of $V_{-a}$. These ray are said \emph{perturbable} if they are adjacent to $V_{-a}$ or adjacent to these ones (see picture \ref{Dyn1}).

The boundary of $V_{-a}$ contains a parabolic point and one pre-image by the first return map on this component. This latter point which is uniquely defined will be called the \emph{co-parabolic point}.
 Given a parabolic ray $\theta$, there exists always a unique ray landing at the co-parabolic point mapped to the same image as ${\cal R}_\theta$. Inspirited by the position of this ray, we will call this ray the \emph{symmetric ray} of $\theta$ and denote its angle by $\overline \theta$ (it follows that $\overline \theta=\theta±1/3$).

Fix some parabolic ray $\theta$. We define by induction a family $(\sim_k)_{k\in \Z}$ of equivalence relations on $\Q/\Z$.
For $k<0$, define $\sim_k$ to be a trivial relation. Define $\sim_0$ to be $\sim_{-1}$ on which we add the relation $\theta\sim_0 \overline \theta$ and $\overline \theta\sim_0  \theta$. Suppose that $\sim_{j-1}$ is well defined. We define $\sim_j$ to be $\sim_{j-1}$ on which we add the relations $\theta_0\sim_j\theta_1$ if and only if $2\theta_0\sim_{j-1}2\theta_1$ and ${\cal R}_{\theta_0}$ and ${\cal R}_{\theta_1}$ land at the boundary of a same Fatou component.

We denote by ${\bf \Theta}_k^\theta$ the set of non trivial classes of $\sim_k$.
\begin{definition}[Lamination of perturbation]
  Given any $k\in\Z$ and for ${\bf \Theta}^k_P=\{\Theta_1,\ldots,\Theta_p\}$, we define the $k$-lamination $\mathcal{L}^\theta_k$ of the perturbation through $\theta$
to be the subset of ${\C}$ obtained as the union of the convex hull of ${\Theta_i \subset \R/2\pi\Z \equiv \partial \D}$ inside $\D$.

We call a $k$-gap  a connected component of the complement of $\mathcal{L}^\theta_k$ in $\D$ and denote by $\mathcal{G}^\theta_k$ the set of $k$-gaps.
\end{definition}

From the definition of the equivalence classes $\sim_k$, there is a natural map $F_{k+1}:\mathcal{L}^\theta_{k+1}\to \mathcal{L}^\theta_k$ that maps a class of ray to the other if and only if the polynomial $P$ does the same for the corresponding set of dynamical rays.

As in section \ref{labelle} there is a natural map $\phi^\theta_k$ that associate to every gap a subset of $\C$. Consider a gap $\gamma_k\in\mathcal{G}_{k}$. The intersection $\overline \gamma_k \cap\partial \D$ consists of a finite union of segments $I_1,\ldots,I_p$. Define $b_k:=\phi^\theta_k(\gamma_k)$ to be the reunion of the rays of angles lying in $\mathring I_1,\ldots,\mathring I_p$. Define $B_k$ to be the union of the $b_k$ for all the $\gamma_k\in\mathcal{G}^\theta_k$. We define $F_{k+1}$ such that the following diagram commutes:

 \centerline{$\xymatrix{
    B_{k+1} \ar[r]^P   & B_k  \\
   \mathcal{G}^\theta_{k+1} \ar[u]^{\phi^\theta_{k+1}}\ar[r]^{F_{k+1}} & \mathcal{G}^\theta_{k}\ar[u]_{\phi_k^\theta}.
  }$}

Again we have defined for every $k\in\Z$ a map $$F_{k+1}:\mathcal{G}_{k+1}^\theta\sqcup\mathcal{L}_{k+1}^\theta\to \mathcal{G}^\theta_k\sqcup \mathcal{L}^\theta_k.$$

\begin{definition}
The family $(\mathcal{L}^\theta_k,F_k)_{k\in\Z}$ is called the dynamical lamination of the perturbation through $\theta$.
\end{definition}


\begin{figure}[h]
\begin{minipage}[c]{.45\linewidth}
\begin{center}

\includegraphics[width=6cm]{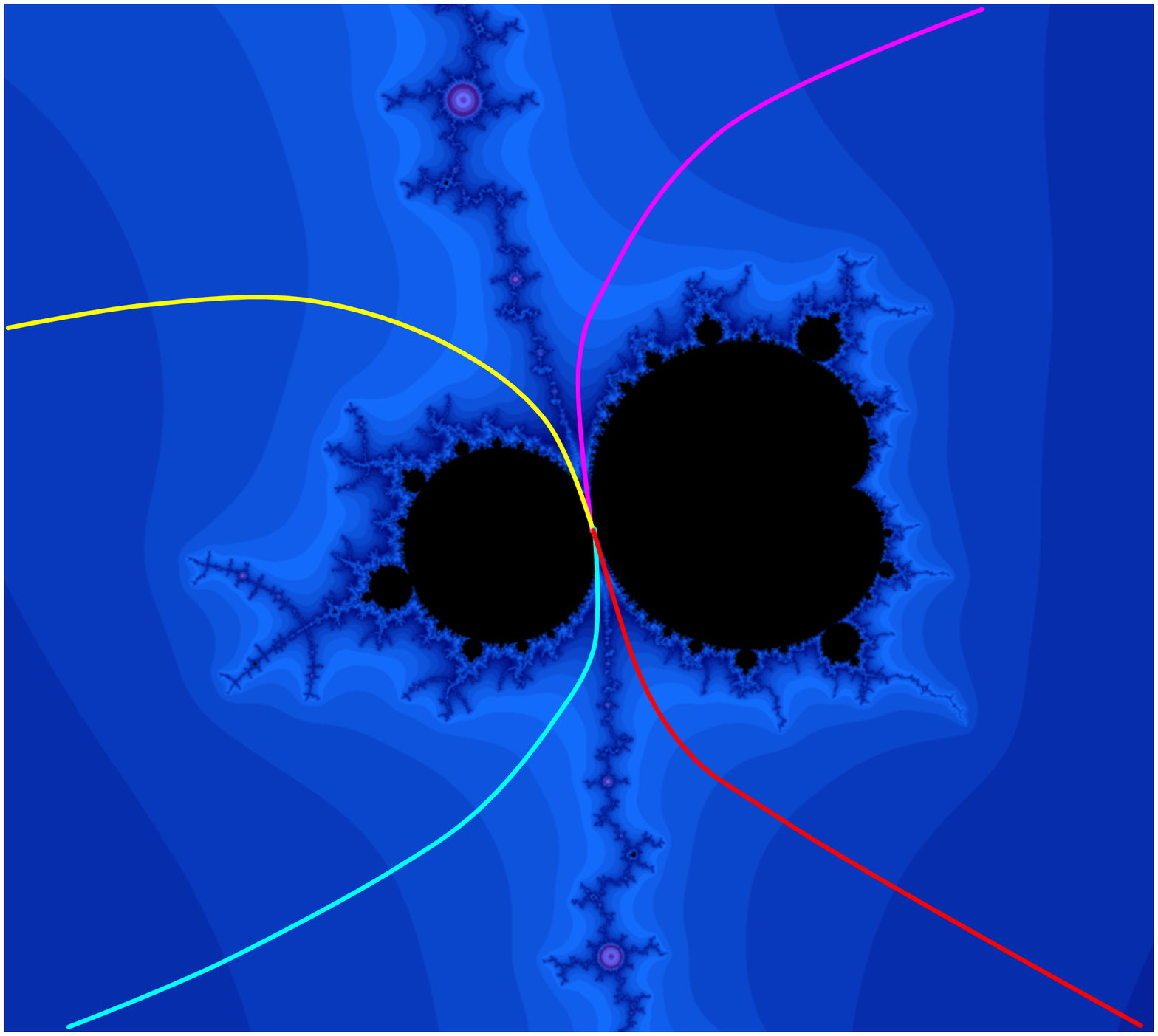}  \caption{A detail of $\St_3$. The escape locus is in blue and ${\cal C}$ in black separating locally two escape components. We colored four different external rays landing to the same points which is an intertwining.}\label{Parm1}

\end{center}
\end{minipage}
\hfill
\begin{minipage}[c]{.45\linewidth}
\begin{center}

\includegraphics[width=6.5cm]{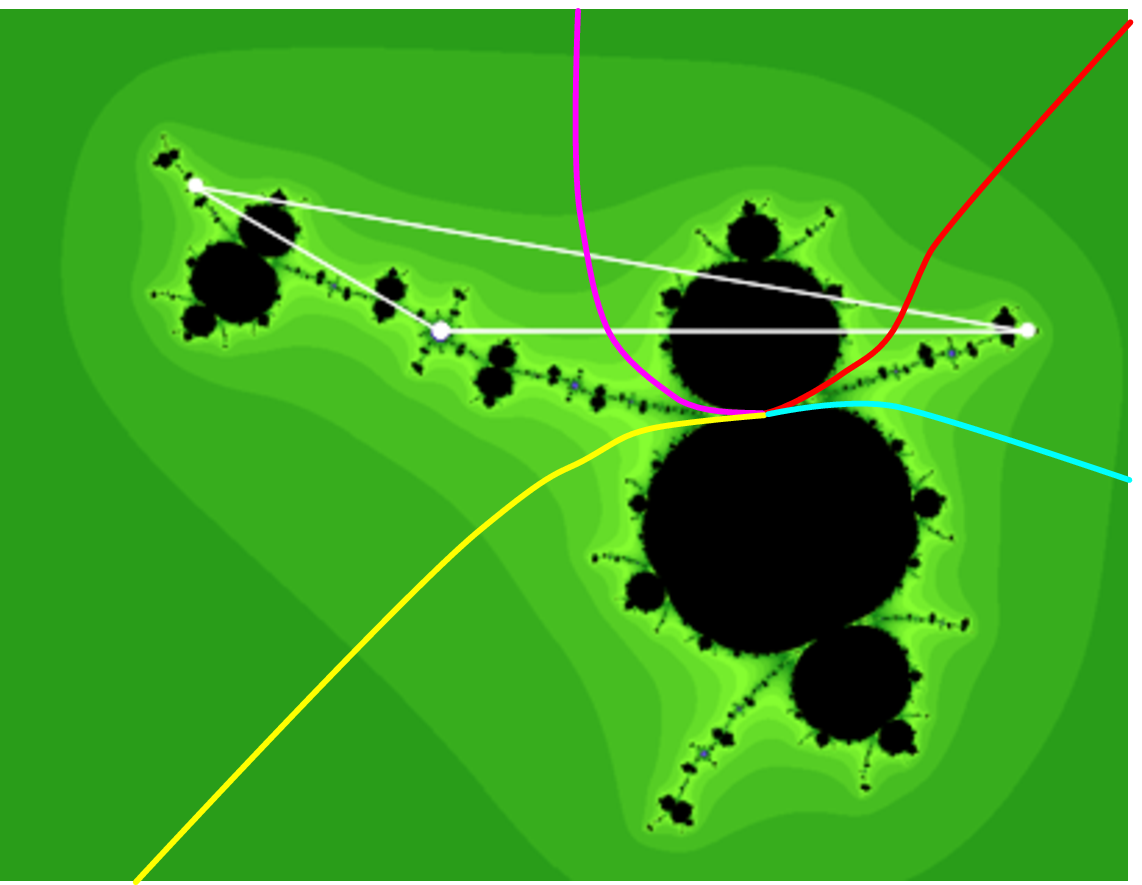}  \caption{The dynamical plane corresponding to the intertwining of Figure \ref{Parm1}. The Julia set is in black. We draw in color the different perturbable rays and in white the periodic critical orbit. Each access allows a perturbation which corresponds to an external angle with the respective color on Figure \ref{Parm1}. }\label{Dyn1}

\end{center}
\end{minipage}
\end{figure}


\subsection{Perturbation.}

In this section we show that we can find the component containing the polynomials obtained by a perturbation of an intertwining. For this we have to provide their escaping lamination (Proposition \ref{corolam}) and their Hubbard tree (Proposition \ref{mainlemma}).

We admit the following theorem:

\begin{theorem}\cite{C2}\label{Cui}
Given a cubic polynomial $P$ with a periodic critical point whose direct basin contains a parabolic point with $k$ accesses to infinity. For every perturbable parabolic ray $\theta$, there exists a canonical path $P_t$ in $\St_n$ such that 
\begin{itemize}
\item $P_0=P$;
\item for $t\neq 0, P_t\notin{\cal C}$;
\item for $t\neq 0, P_t$ the external rays $\theta$ and its symmetric ray $\overline\theta$ are the critical rays.
\end{itemize}
\end{theorem}

\begin{definition}
Such a family $(P_t)_t$ is called a perturbation of $P$ along $\theta$.
\end{definition}

\begin{remark}\label{coroglaire} It follows from Lemma \ref{sameangle} that any path given by Theorem \ref{Cui} is contained in the closure of a external ray.  Figure \ref{Parm1} represents the four parameter rays corresponding to the four perturbable rays on Figure \ref{Dyn1}.
\end{remark}

\begin{proposition}\label{corolam} With the same notations, for all $t>0$ and every $k\in\Z$, the dynamical escaping lamination of $P_t$ is the dynamical lamination of the perturbation through $\theta$.
\end{proposition}

\begin{proof}
From Theorem \ref{Cui} we deduce that the laminations are the same. 
Consider a $k$-gap $\gamma$ of ${\cal L}^\theta_k$ for some $k\in\Z$. There exists a dense set of rays in $\gamma$ of rational angle with land at a repelling periodic point of $P_0$. According to Lemma \ref{Stability}, we can follow continuously the landing point of these rays along the family $P_t$ for $t$ small and the dynamics is preserved. It follows that the dynamics of the gaps is preserved.
\end{proof}

Following \cite{Ha}Appendice A we define:

\begin{definition}\label{defcopy} Let $P$ and $Q$ be polynomials with a connected filled Julia sets. We say that $Q$ has a rigid figuration in $P$ if there exists a neighborhood $U$ of $K_Q$ and $\phi:U\to\C$ continuous and injective such that $\phi\circ Q=P\circ \phi$ on $K_Q$ and such that:
\begin{itemize}
\item $\phi$ and $\phi^{-1}$ preserve the zero mesure sets;
\item $\phi\in{\cal W}^{1,p}$ for some $p>1$;
\item $\overline\partial \phi=0$ almost everywhere on $K_Q$.
\end{itemize}
\end{definition}

We will say that $\phi(K_Q)$ is the filled Julia set of $Q$ in $P$.

\begin{remark}\label{remrep}
As remarked in \cite{Ha} (Proposition A1), we note that with this definition the image by $\phi$ of a repulsive point is a repulsive point. 
\end{remark}

Recall that if $P\in\St_n$ is a parabolic polynomial and ${\cal R}_\theta$ a perturbable ray of $P$, then ${\cal R}_\theta$ and ${\cal R}_{\overline \theta}$ land by definition at the boundary of the same Fatou component $V_{-a}$ containing the critical point $-a$.

\begin{definition}
Assume that there exists a rigid figuration of $Q$ in a parabolic polynomial $P\in\St_n$. We say that a perturbable ray $ {\cal R}_\theta$ cut this figuration if the orbit of $a$ inside the filled Julia set of $Q$ in $P$ visits the two components of $\C\setminus( {\cal R}_{\theta}\cup \overline{V_{-a}}\cup {\cal R}_{\overline\theta}).$
\end{definition}

\begin{proposition}\label{mainlemma}
Let $P$ be a parabolic polynomial in $\St_n$ and let $(P_t)_t$ be a perturbation of $P$ along one of its perturbable rays $\theta$. 

If for some $k\geq 1$ that we suppose to be minimum, there exists a quadratic polynomial $Q$ with a periodic critical point that has a rigid figuration in $P^k$ and if $\theta$ is cutting this figuration
then, for $t>0$ small enough, $Q$ is the connected restriction of $P_t$.
\end{proposition}

To prove this result we use the following lemma whose proof is left to the reader and follows from the general study of geodesic laminations (see \cite{T} for example).

\begin{lemma} \label{geodlam}
If two quadratic polynomials with a periodic critical point are not conjugated by a translation then there exist two rational rays landing at the same point for the first one which are not for the second one.
\end{lemma}

\begin{proof}(of Proposition \ref{mainlemma})
It follows from Proposition \ref{corolam} 
that, for $t>0$, the escaping lamination of $P_t$ has an end of exact period $k$. Thus it has a connected restriction $Q_t$ which has a critical point of period $n/k$ as $Q$ does. Without loose of generality we can suppose that $Q$ lies in the family $P_c:z\mapsto z^2+c$ with $c\in \C$. It remains to prove that $Q_t=Q$. We denote by $\hat K_Q$ the filled Julia set of $Q$ in $P$ and by $\hat K_{t}$ the filled Julia set of ${Q_t}$ in $P_t$.

On the one hand, one can follow the repulsive cycles of $\hat  K_Q$ in a small neighborhood of $P$. Lemma \ref{Stability} assures that the rays landing at these cycles can be also followed, so the repulsive cycles of $\hat  K_Q$ become repulsive cycles for $\hat  K_t$ with same period.

On the second hand, a global counting of the repulsive cycles for quadratic polynomials with a period $n/k$ critical point allows to deduce that, the perturbation of the parabolic cycle of $\hat  K_Q$ create a repulsive cycle of $\hat  K_t$ of exact same period (for the first return map on $\hat  K_t$). 
The rays landing at the parabolic cycle are periodic rays so they have to land to a cycle of a period dividing their periods. They don't have any other choice than landing at the repulsive cycle of $\hat  K_t$ created by the perturbation of the parabolic one. 

In addition, if two rays landing at the same parabolic point on $\hat  K_Q$ lands at different points of the new repulsive cycle, then, considering that they would be separated by a pair of ray in $\hat  K_t$ landing at a same repulsive point in a different cycle, they would have been already separated by this pair in $\hat  K_Q$ which is a contradiction. 

We conclude the proof by applying Lemma \ref{geodlam}. 

\end{proof}

Note that one could certainly prove that there is a continuous motion of the figuration all along the perturbation, however this is not necessary for our purpose.


\section{Parabolic intertwinings and perturbation}\label{Chap4}

\subsection{Construction and existence.}

In this section we recall heuristically the construction of parabolic intertwinings by  P. Haïssinsky in \cite{Ha} as a generalization of \cite{EY}.

Consider two polynomials: $Q$ and $R$ of degree $2$ with connected Julia set. Take a non attractive fix point $\alpha_Q$ of $Q$ and a non attractive periodic point $\alpha_R$ of $R$ such that $\alpha_Q$ and $\alpha_R$ have the same rotation number $p/q$. Choose a dynamical ray landing at $\alpha_Q$. We will call this ray the principal access of this construction. 
Cut both of the basins of infinity along the dynamical rays landing at points in the orbit of $\alpha_Q$ and $\alpha_R$. 
After identifying $\alpha_Q$ and $\alpha_R$, glue carefully the two polynomials along a "sectorial neighborhood" of these rays (see Figure \ref{intertw0}) such that the critical point of $R$ lies in the gap created by cutting along the principal access. Do consistently the same gluing  all along the backward orbit of  $\alpha_R$. Adding one point at infinity, we define a topological branched cover of the sphere with a degree $3$ fixed critical point. When this topological polynomial is conjugated to a conformal map, we denote that latter one by $Q\intertw R$ and call it the intertwining of $Q$ and $R$.

 \begin{figure}
  \centerline{\includegraphics[width=12cm]{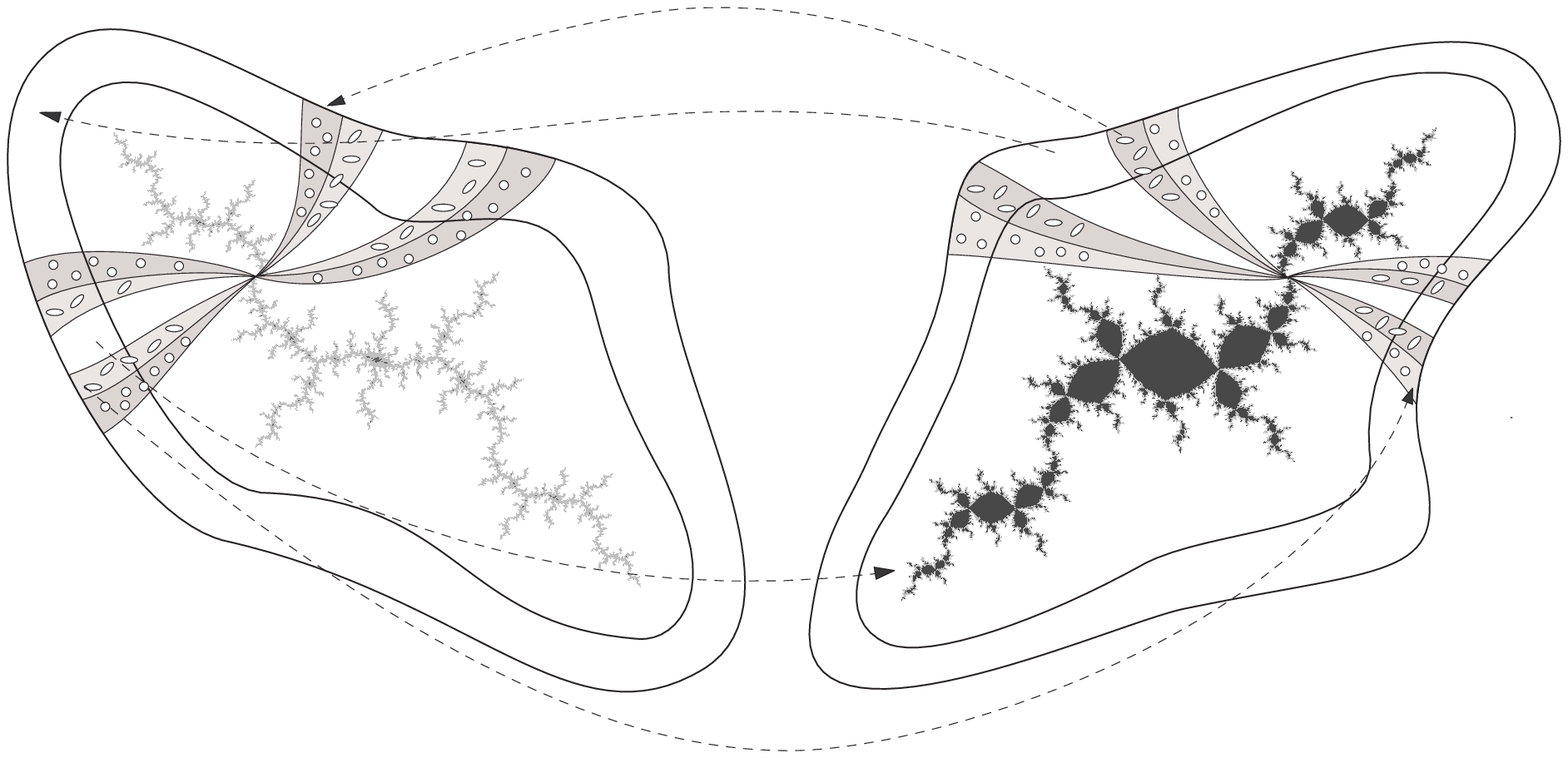}} 
   \caption{Figure from \cite{EY}.}
\label{intertw0} \end{figure}

In \cite{EY}, A.Epstein and M.Yampolsky proved the existence of $Q\intertw R$ for all possible choice of principal access when both $\alpha_Q$ and $\alpha_R$ are not parabolic. P.Haïssinsky in \cite{Ha} extended their results to a more general case allowing one of these two points to be parabolic:

\begin{theorem}If $Q$ or $R$ is hyperbolic then the polynomial ${Q\intertw R}$ is uniquely determined by $\alpha_Q,\alpha_R$ and the principal access landing at $\alpha_Q$. In addition, $Q$ and $R$ have a rigid figuration in 
${Q\intertw R}$.
\end{theorem}

\begin{figure}[h]
\begin{minipage}[c]{.45\linewidth}
\begin{center}
\includegraphics[width=6cm]{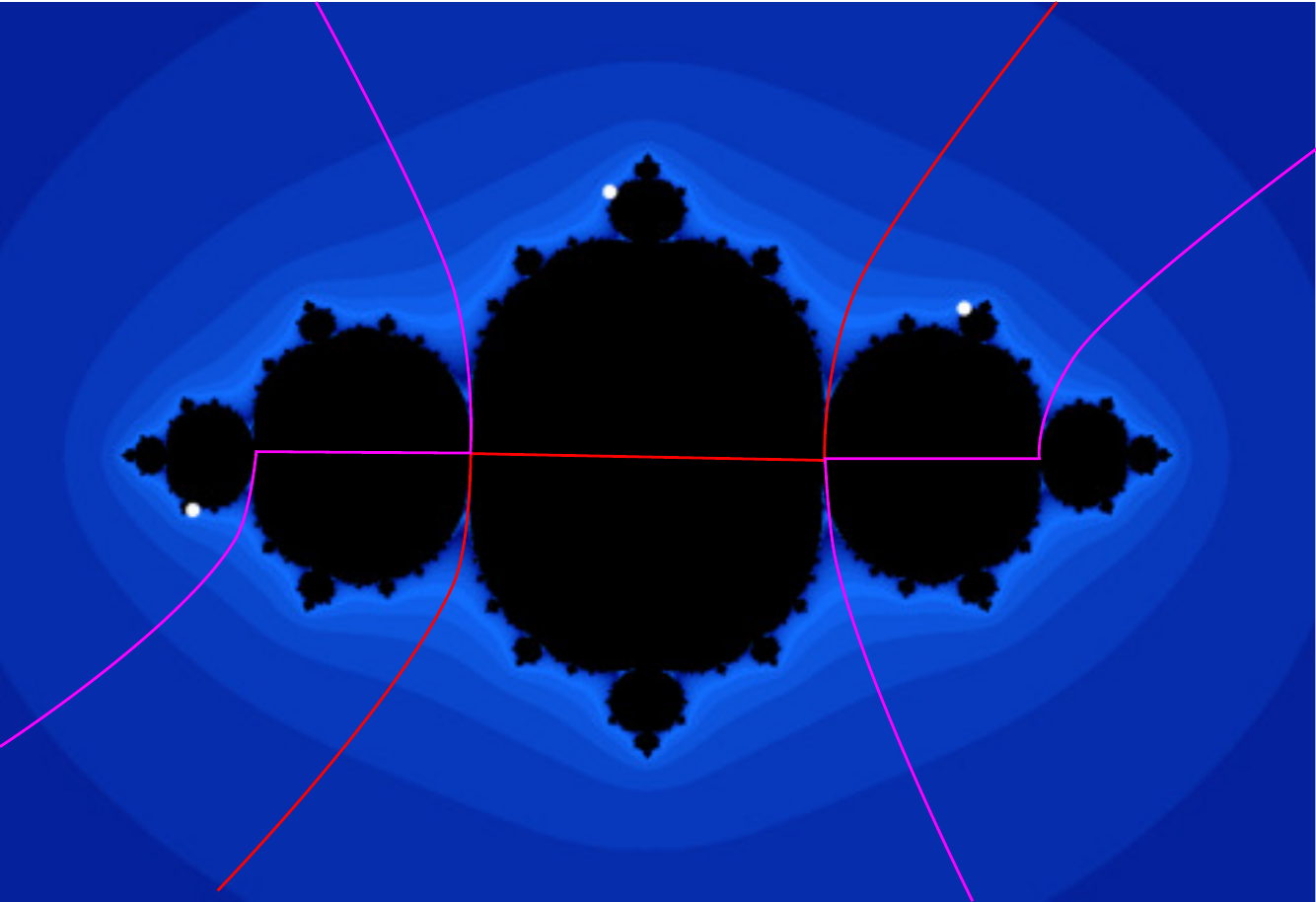} 
\end{center}
\end{minipage}
\hfill
\begin{minipage}[c]{.45\linewidth}
\begin{center}
\includegraphics[width=6cm]{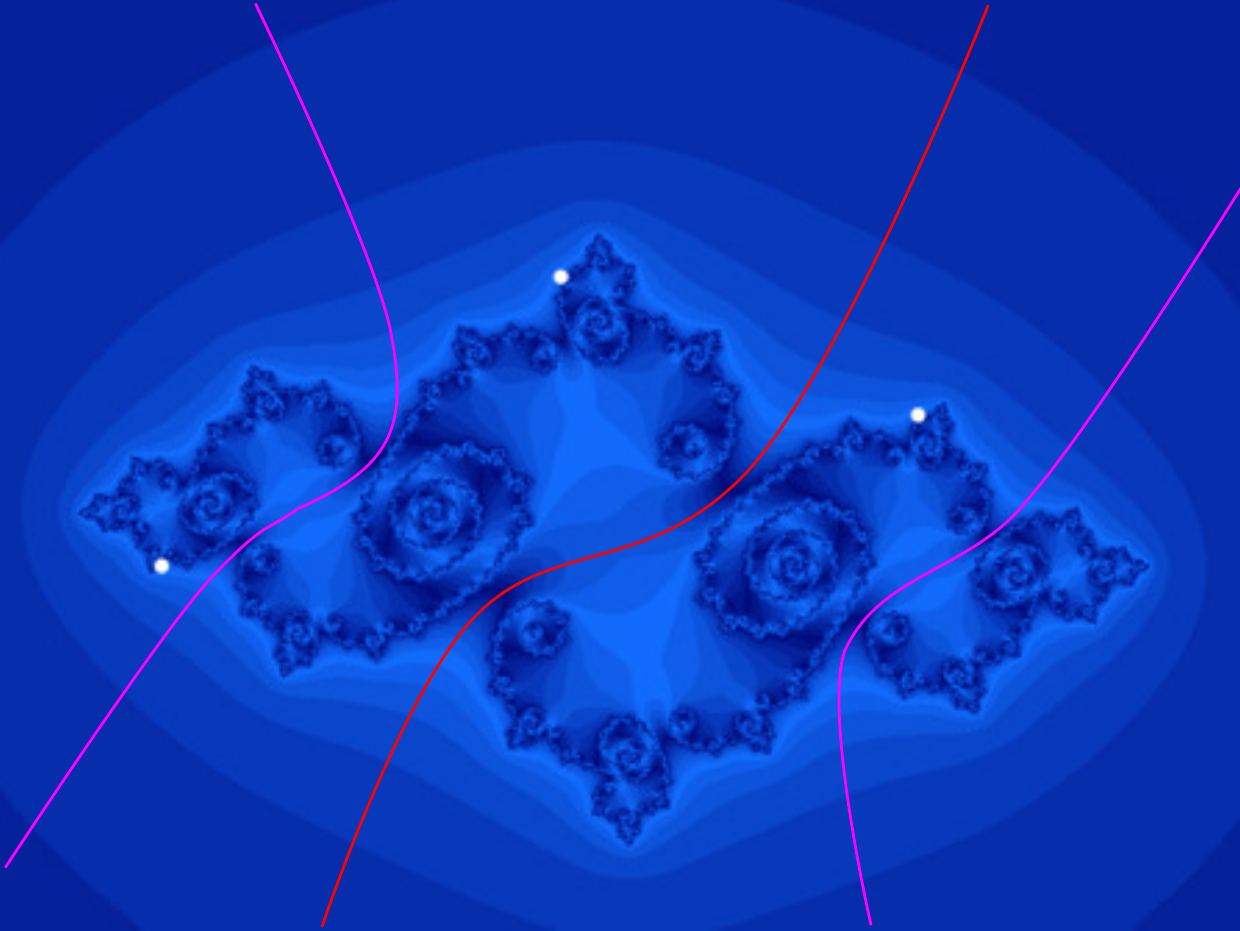} 
\end{center}
\end{minipage}

\begin{minipage}[c]{.45\linewidth}
\begin{center}
\includegraphics[width=6cm]{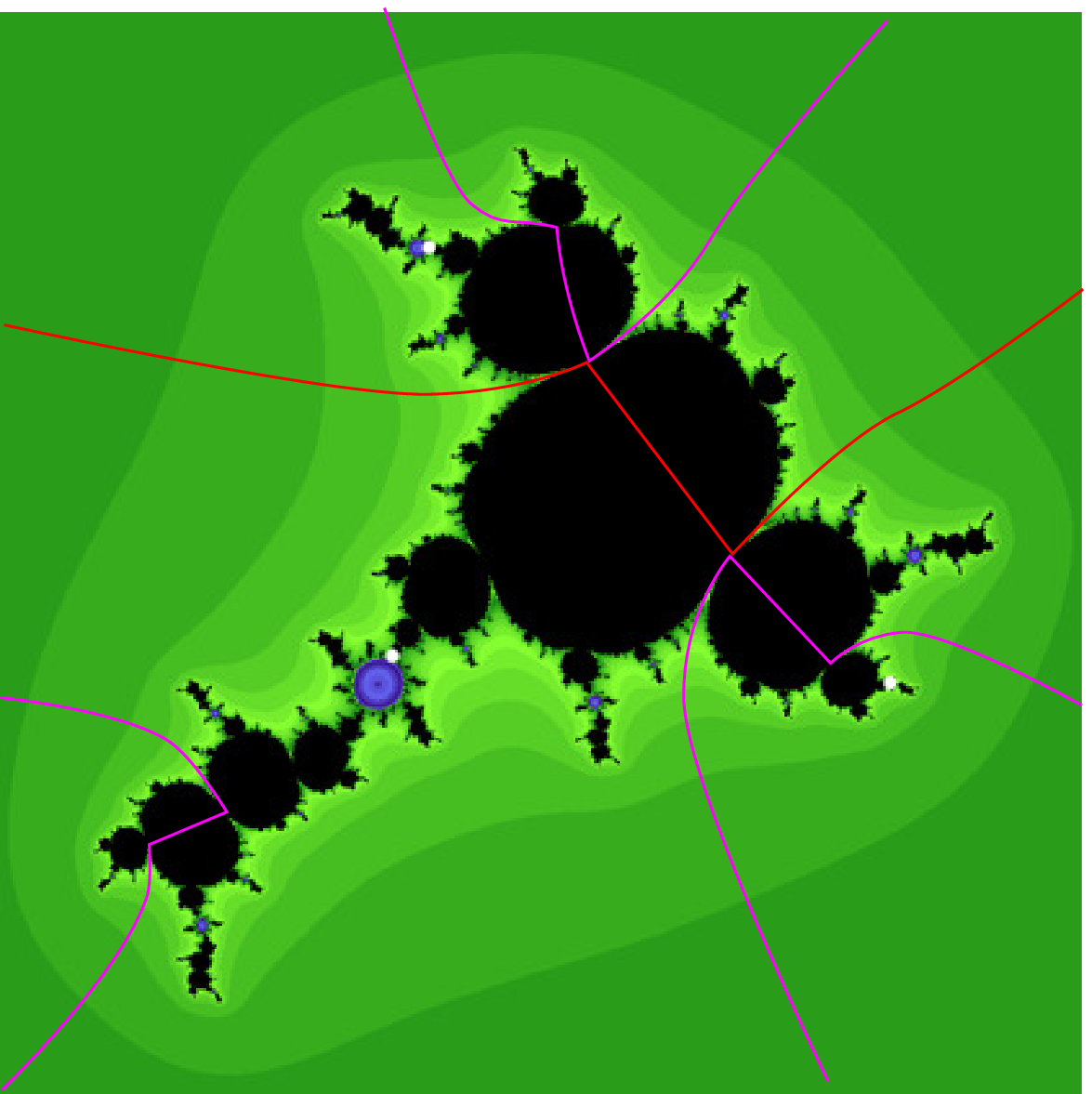} 
\end{center}
\end{minipage}
\hfill
\begin{minipage}[c]{.45\linewidth}
\begin{center}
\includegraphics[width=6cm]{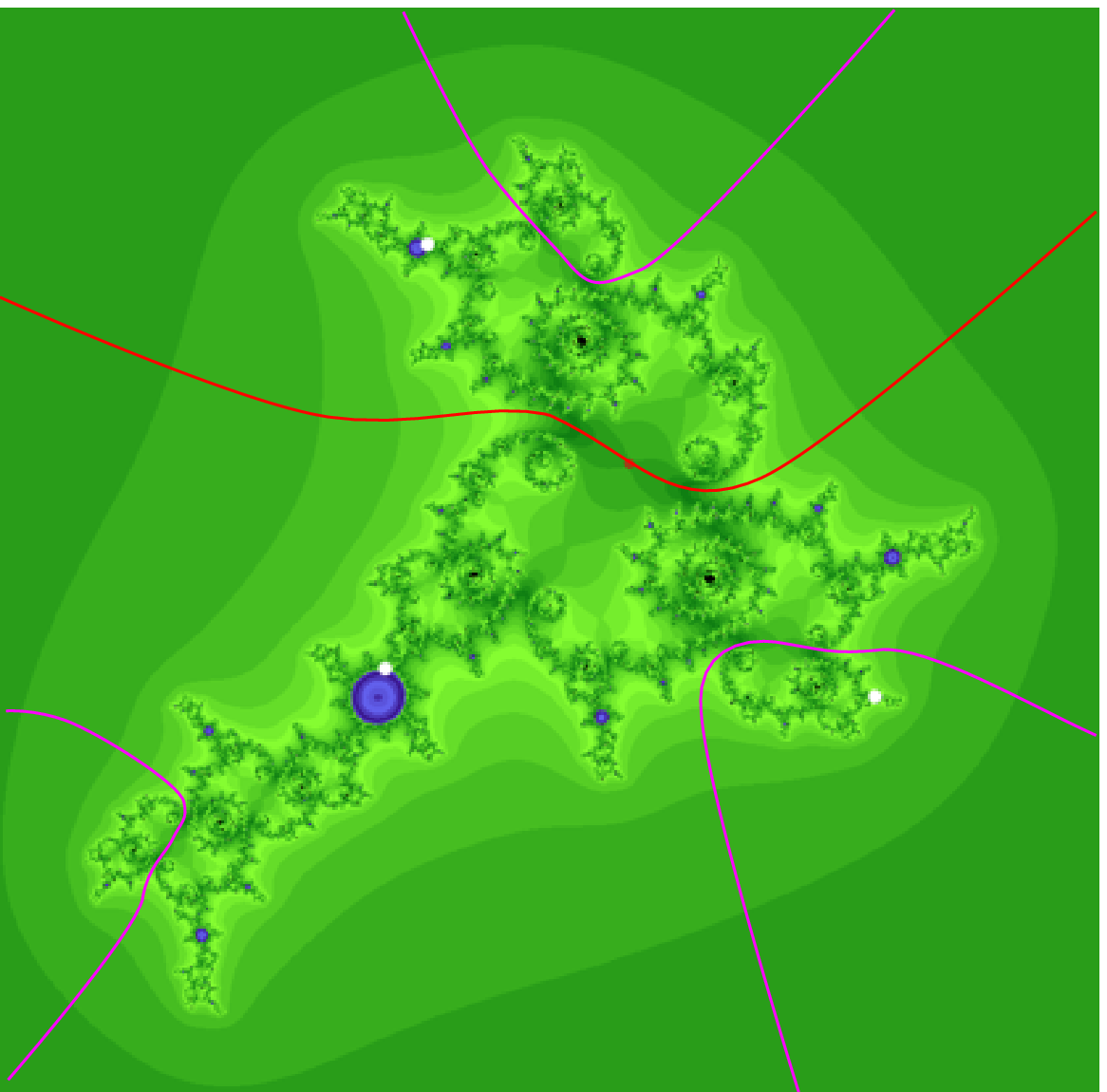} 
\end{center}
\end{minipage}\caption{The two blue pictures show a parabolic perturbation of a quadratic polynomial with the angle involved in the corresponding escaping lamination. The two green pictures show the same perturbation for the same polynomial intertwined at a period 3 cycle represented in white on all the pictures. On the left the white dots represent the orbit of the parabolic periodic point and on the right one, it represents the corresponding repulsive cycle for the perturbed maps.}\label{3}
\end{figure}

In this paper we will be interested in the following cases:

\begin{itemize}
\item {\bf Type 1:} the point $\alpha_R$ is a $n$-periodic point of a polynomial $R$ with a super-attractive cycle of period $n$ and $\alpha_Q$ is a fixed parabolic point of $Q$ (see Figure \ref{Dyn1});
\item {\bf Type 2:} the point $\alpha_Q$ is a $n$-periodic point of a parabolic polynomial $Q$ and $R$ is the polynomial $z^2$ (see the left below picture on Figure \ref{3}).
\end{itemize}


\subsection{Reading escaping trees from Hubbard trees.}\label{reading}

Using Proposition \ref{corolam} and Proposition \ref{mainlemma}, we find the escaping tree and the Hubbard tree of the escape region corresponding to a perturbation of an intertwining along a perturbable ray.

\medskip
\noindent{\textbf{Type 1.}}
Consider a Type 1 intertwining $P$ such that $\alpha_R$ lies on the Hubbard tree of $R$ and $\alpha_Q$ is the parabolic point which direct basin contains a critical point. (Note that it follows form the general studies of the Mandelbrot set, that can always find such a polynomial $Q$ to provide such a $P$.)

By construction there exist alway two rays adjacent the the Fatou component containing the critical point of infinite orbit (see the rays yellow and cyan on Figures \ref{Parm1} and \ref{Dyn1} for example). By definition, these two rays are perturbable. The associated perturbations give a polynomial with a trivial kneading sequence and according to Proposition \ref{mainlemma}, the Hubbard tree associated to the corresponding escape component is the one of $R$.

Consider now any other perturbable ray ${\cal R}_\theta$ (see the ray purple or red on Figures \ref{Parm1} and \ref{Dyn1} for example). In this case $ {\cal R}_\theta$ cut the figuration of $R$ thus the perturbation along it gives an escape component with a non trivial kneading sequence.
One reconstruct the corresponding simplified escaping lamination by taking the iterated preimages of ${\cal R}_{\theta}\cup \overline{V_{-a}}\cup {\cal R}_{\overline\theta}$ and by looking when they separate in different connected components the orbit of $a$.

Recall that a kneading sequence is trivial if it is composed only with 0s. As for $n>1$, we can always find rays landing as described upper, we have shown the following result.

\begin{proposition}
Given an escape component with a trivial kneading sequence, there exists an escape component with non trivial kneading sequence such that the two of them share a point of their boundary. 
\end{proposition}

It is easy to check that two different intertwinings on the sam branch of the Hubbard tree give the same connexions between escape components. Hence, as there are finitely many  quadratic polynomial with a period $n$ critical point, it is possible to enumerate all the possible connections obtained this way. 
We give in Figures \ref{class4} and \ref{class5} a list of the connexions obtained this way for $n=4$ and $n=5$. The complete lists are easy to recover by adding the few missing Hubbard trees of quadratic Julia sets: on Figure \ref{class4} one just have to consider the conjugated ones and on Figure \ref{Mandel5} are missing the conjugated and those which are topologically like the D but with an other rotation number around the central vertex (which will give the same relations as those found for D). {\bf This list is sufficient to deduce Theorem \ref{main}}, ie the connectivity of $\St_4$. 

\begin{remark}Note that our results are consistent for $n=4$ with the observation in \cite{CubPol2} and that it gives the exhaustive list of escaping components at the boundary of the one with a trivial kneading sequence.
We conjecture that this process gives such an exhaustive list for all the periods.
\end{remark}

We remark also that, from our description, it is sufficient to know the dynamics of the quadratic polynomial $Q$ on its Hubbard tree in order to give the possible connexions trough $P$. Hence, our process allows to produce a lot of cubic escaping trees just from quadratic Hubbard trees.

In addition, if one consider a Type 1 intertwining for the period $n$ repulsive point with is at the boundary of the Fatou component containing $a$ (resp. $P^k(a)$ for $k<n$) then according to  \cite{C2} one can perturb it in order to go to an hyperbolic component Type A (resp. Type B) ie a polynomial with the two critical points in the same Fatou component (resp. Fatou component in the same periodic cycle). So we deduce the following proposition:

\begin{proposition}
There are at least $n-1$ Type B and $1$ Type A hyperbolic components sharing points of their boundary with the boundary of each escape component of trivial kneading sequence.
\end{proposition}

\medskip
\noindent{\textbf{Type 2.}}
For the Type 2 intertwinings, the configuration is more complicated because there are infinitely many quadratic polynomials with a parabolic fixed point. There would be a lot to say about which configurations give the same connections but we keep these descriptive remarks for a later publication.

Something to be noted is that, when we perturb a Type 2 intertwining, there is always from our construction a gap of ${\cal L}_1$ for all the $P_t$ with $t>0$ which does not contain any iterate of the periodic critical point. We can deduce that any escape component whose escaping laminations wouldn't have this property does not have a Type 2 intertwining in its boundary. This is the case for the escape component with the tree Figure \ref{Misstree} in $\St_5$. As this component is not connected with an escape component of trivial kneading sequence by a Type 1 intertwining, we proved:

\begin{proposition}\label{prob5}
In $\St_5$, there exists an escape component without Type 1 or Type 2 intertwining in its boundary.
\end{proposition}

As we are only looking at the simplified escaping trees, this last remark shows that the only pair of rays that will give an interesting  non trivial equivalence relation for the perturbed map lands on the copy of the parabolic quadratic polynomial (cf Figure \ref{3}). Hence, again a good knowledge of the dynamics of this parabolic quadratic polynomial is sufficient to produce cubic polynomial escaping trees.

 \begin{figure}
  \centerline{\includegraphics[width=5cm]{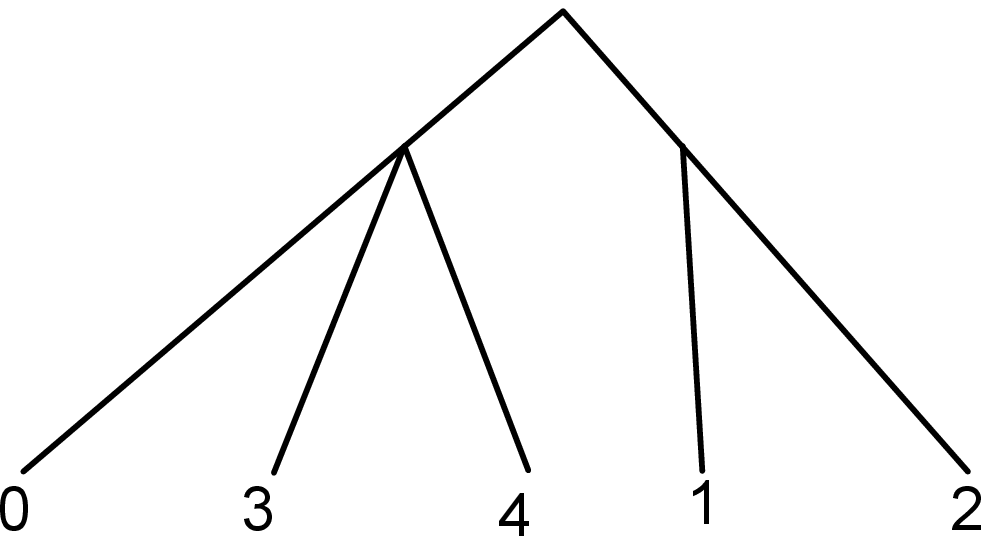}} 
   \caption{This tree cannot be obtained by a perturbation of Type 2 intertwining. It's kneading sequence is 11000.}
\label{Misstree} 
\end{figure}

\medskip
\noindent{\textbf{About $n\geq5$.}}
There are two problems to use our approach for $\St_n$ with $n\geq5$. The first is that Proposition \ref{DPclass} is false for $n=5$ and has no reason to be true for higher periods. For $n=5$, both of the escaping trees with respective kneading sequences 10110 and 11010 (see Figure \ref{class5}) are associated to two different escaping components (this can be proven using the algorithm from \cite{DS2}). However, the escaping laminations contain much more informations that the escaping trees, thus we could use these latter ones to differentiate escaping components with same escaping trees.

The second problem however cannot be solved, this problem has been pointed out in Proposition \ref{prob5}. Note that using intertwined polynomial is very restrictive and that our process can be applied to a more general type of polynomials (cf parabolic polynomials section \ref{sectionparab}). This idea has been discussed with J.Kiwi and we found a polynomial in the boundary of a type A hyperbolic components (both of the critical point in the same Fatou component) that would gives us a connection between the escape component associated with this tree and an other one.

 \begin{figure}
  \centerline{\includegraphics[width=11.5cm]{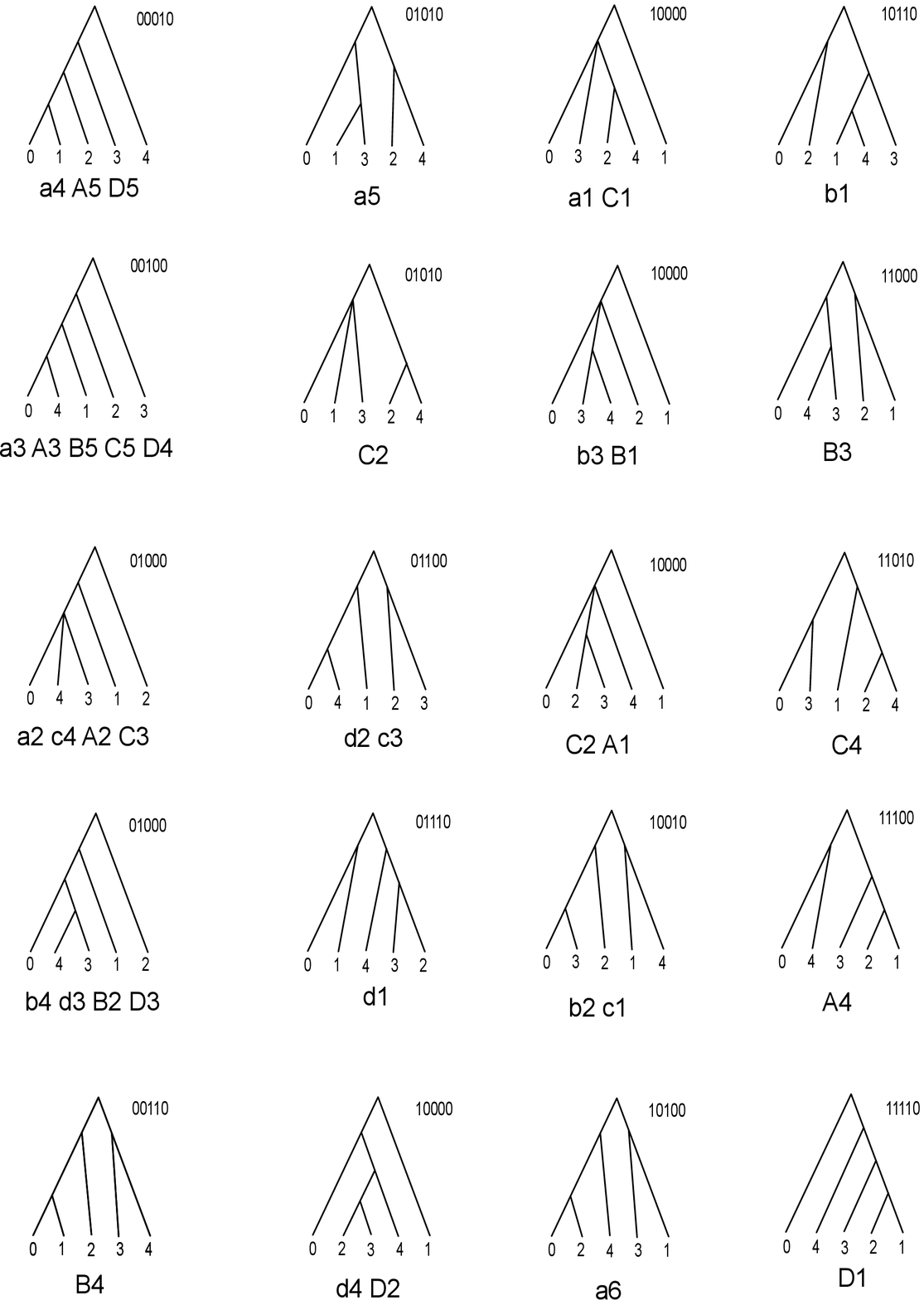}} 
   \caption{All possible simplified escaping trees for polynomials in $\St_5$ obtained by perturbation of intertwining (ie all but the trivial tree and the one on Figure \ref{Misstree}). Below each tree is written the label of an edge of Hubbard tree of Figure \ref{Mandel5}. This edge contains the landing point of the corresponding intertwining's principal access.}
\label{class5} 
\end{figure}

 \begin{figure}
  \centerline{\includegraphics[width=11.5cm]{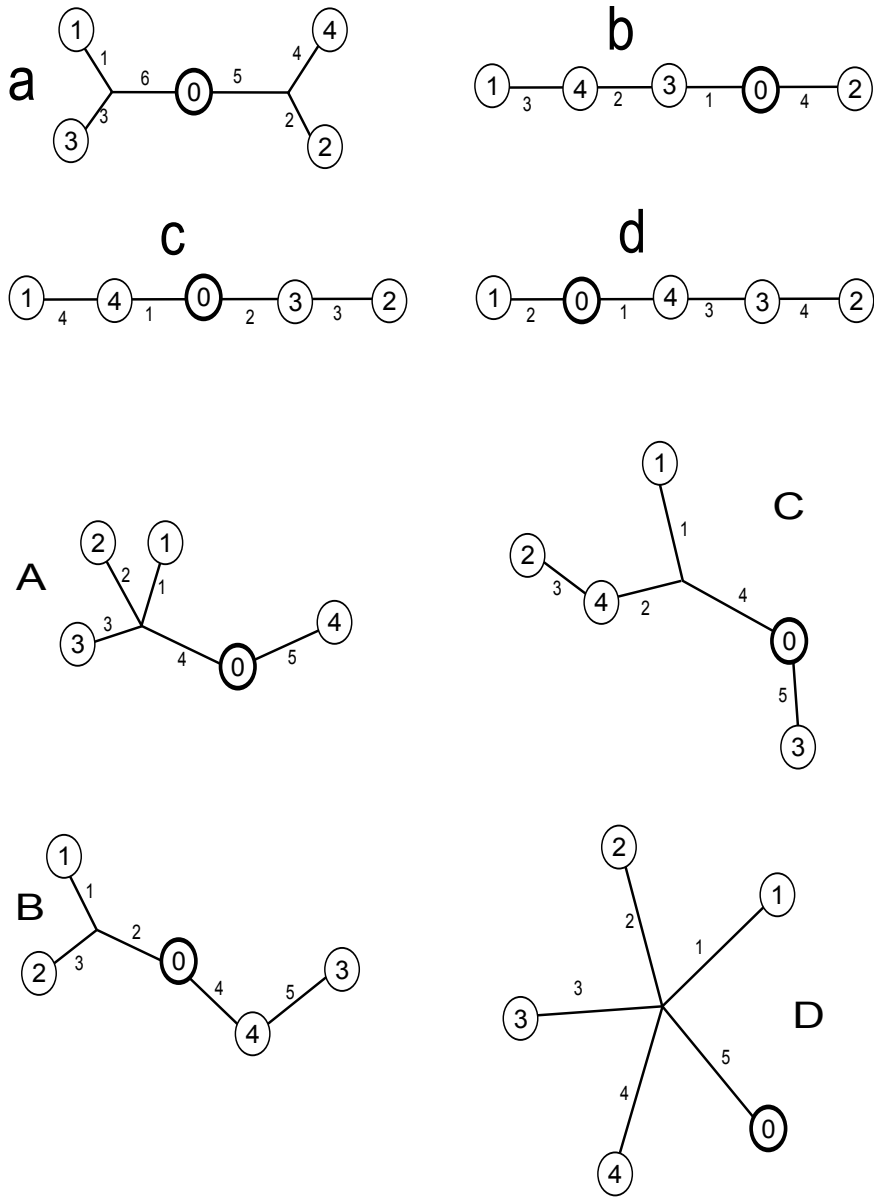}} 
   \caption{Enumeration of some Hubbard trees for quadratic polynomials with a period 5 critical point. We affected a random label to every edge.}
\label{Mandel5} 
\end{figure}


\bibliographystyle{alpha}
\bibliography{biblio}

\end{document}